\documentclass[12pt, abstracton]{scrartcl}   

 \usepackage[T1]{fontenc}
      \input{dehyphtex}              
 \usepackage[width=16cm, height=22cm]{geometry}   
 \usepackage{amsmath,amssymb}         
 \usepackage[amsmath,thmmarks,hyperref]{ntheorem} 
 \usepackage{icomma}             
 \usepackage{enumerate}             
 \usepackage{mathrsfs}         
 \usepackage{slashed}
 \usepackage{color}
  \usepackage{esint}

 \numberwithin{equation}{section}
 
 \usepackage[numbers,square]{natbib}
\theoremstyle{nonumberplain}  
\theoremheaderfont{\itshape}  
\theorembodyfont{\normalfont}  
\theoremseparator{.}  
\theoremsymbol{$\Box$}
\newtheorem{proof}{Proof} 
  
\theoremstyle{plain}  
\theoremheaderfont{\normalfont\bfseries}  
\theorembodyfont{\itshape}  
\theoremsymbol{~}
\theoremseparator{.}  
\newtheorem{proposition}{Proposition}[section]  
  
\newtheorem{lemma}[proposition]{Lemma}  
\newtheorem{theorem}[proposition]{Theorem}   

\theorembodyfont{\normalfont}  
 
\newtheorem{remark}[proposition]{Remark}
\newtheorem{example}[proposition]{Example}  
\newtheorem{nonexample}[proposition]{Non-Example}  
  
\newtheorem{definition}[proposition]{Definition} 
\newtheorem{notation}[proposition]{Notation} 

\theoremstyle{nonumberplain}
\theorembodyfont{\normalfont}

\newcommand*{\res}{\operatorname{res}}
\newcommand{\R}{\mathbb{R}}

\newcommand{\V}{\mathcal{V}}
\newcommand{\W}{\mathcal{W}}
\newcommand{\N}{\mathbb{N}}
\newcommand{\Z}{\mathbb{Z}}

\newcommand{\dd}{\mathrm{d}}

\newcommand{\tr}{\mathrm{tr}}
\newcommand{\End}{\mathrm{End}}

\newcommand{\id}{\mathrm{id}}

\newcommand{\<}{\left\langle}
\renewcommand{\>}{\right\rangle}
\newcommand{\n}{\mathbf{n}}

\title{Path Integrals on Manifolds with Boundary}
\author{ Matthias Ludewig}

\begin{document}

\maketitle
 
\begin{center}
  Max Planck Institut f\"ur Mathematik \\ 
  Vivatgasse 7 / 53111 Bonn, Germany \\ \medskip
 maludewi@mpim-bonn.mpg.de
\end{center}

\begin{abstract}
We give time-slicing path integral formulas for solutions to the heat equation corresponding to a self-adjoint Laplace type operator acting on sections of a vector bundle over a compact Riemannian manifold with boundary. More specifically, we show that such a solution can be approximated by integrals over finite-dimensional path spaces of piecewise geodesics subordinated to increasingly fine partitions of the time interval. We consider a subclass of mixed boundary conditions which includes standard Dirichlet and Neumann boundary conditions.
\end{abstract}

%%%%%%%%%%%%%%%%%%%%%%%%%%%%%%%%%%%%%%%%%%%%%%%%%%%%%%%%%%%%%%%%

\section{Introduction}

First let $M$ be a compact Riemannian manifold without boundary and let $V \in C^\infty(M)$. For $u_0 \in L^2(M)$, let $u$ be the solution to the heat equation
\begin{equation*}
  \left(\frac{\partial}{\partial t} + \Delta + V\right)u(t, x) = 0, ~~~~~~~~ u(0, x) = u_0(x).
\end{equation*}
It is a well-known heuristic physicist's principle going back to Feynman \cite{FeynmanHibbs} that the solution $u(t, x)$ should be given by the {\em path integral}
\begin{equation} \label{HeuristicPathIntegral}
  u(t, x) = \fint \exp \left( - \frac{1}{4}\int_0^t \bigl|\dot{\gamma}(s)\bigr|^2 \dd s - \int_0^t V\bigl(\gamma(s)\bigr) \dd s \right) u_0\bigl(\gamma(t)\bigr) \mathcal{D}\gamma
\end{equation}
where one integrates over the space of all paths $\gamma: [0, t] \longrightarrow M$ satisfying $\gamma(0) = x$ with respect to some Lebesgue type measure on this space (and normalizes suitably, indicated by the slash over the integral sign). However, it is somewhat difficult to make sense of formula \eqref{HeuristicPathIntegral} mathematically: It is not clear which regularity the paths should have (although the first term indicates that they should have at least the first derivative in $L^2$), but for most options, it is known that there does not exist a measure $\mathcal{D}\gamma$ such that \eqref{HeuristicPathIntegral} holds. To make things worse, the normalization constant will be usually infinite.

%One way to make sense of formula \eqref{HeuristicPathIntegral} is the Wiener measure. Namely, by the well-known Feynman-Kac formula, one has
%\begin{equation*}
%  u(t, x) = \int \exp \left( - \int_0^t V\bigl(\gamma(s)\bigr) \dd s \right) u_0\bigl(\gamma(t)\bigr) \dd\mathbb{W}(\gamma)
%\end{equation*}
%where one integrates with respect to the Wiener measure on the space of all continuous paths $\gamma: [0, t] \longrightarrow M$ with $\gamma(0) = x$. Notice however, that the integral over $|\dot{\gamma}(s)|^2$ disappeared; it has been absorbed into the measure. This may be disadvantageous for applications.

One way make sense of formula \eqref{HeuristicPathIntegral} is by approximating the space of all paths by finite-dimensional spaces of geodesics. Namely, for a partition $\tau = \{0 < \tau_0 < \tau_1 < \dots < \tau_N = t\}$ of the interval $[0, t]$, set
\begin{equation} \label{DefinitionHTau}
  H_{x;\tau}(M) := \bigl\{ \gamma \in C([0, t], M) \mid \gamma|_{[\tau_{j-1}, \tau_j]}~\text{is a geodesic} \bigr\}.
\end{equation}
We will see that this space has a natural manifold structure. Tangent vectors at a path $\gamma \in H_{x;\tau}(M)$ can be naturally identified with piece-wise Jacobi fields along $\gamma$, and it turns out that the discretized $H^1$ metric
\begin{equation} \label{DiscretizedMetric}
  (X, Y)_{\Sigma\text{-}H^1} := \sum_{j=1}^N \< \frac{\nabla}{\dd s}X(\tau_{j-1}+), \frac{\nabla}{\dd s}Y(\tau_{j-1}+)\> \Delta_j \tau
\end{equation}
is the most natural Riemannian metric on $H_{x;\tau}(M)$ (here, the $+$ indicates that we take the derivative coming from above at the node, and we set $\Delta_j\tau := \tau_j - \tau_{j-1}$). Integrating over $H_{x;\tau}(M)$ with respect to the Riemannian volume measure for this metric, one has
\begin{equation} \label{FiniteDimensionalApproximation}
  u(t, x) = \lim_{|\tau|\rightarrow 0} \fint_{H_{x;\tau}(M)} \exp \left( - \frac{1}{4}\int_0^t \bigl|\dot{\gamma}(s)\bigr|^2 \dd s - \int_0^t V\bigl(\gamma(s)\bigr) \dd s \right) u_0\bigl(\gamma(t)\bigr) \dd\gamma,
\end{equation}
where the limit goes over any sequence of partitions with mesh $|\tau| := \max_{1\leq j \leq N} \Delta_j \tau$ going to zero, and the slash over the integral sign indicates that we normalize the integral dividing by $(4\pi)^{-\dim(H_{x;\tau}(M))/2}$. Formula \eqref{FiniteDimensionalApproximation} is valid uniformly in $x$ in the case that $u_0$ is continuous and holds in the $L^p$ sense in the case that $u \in L^p$.
For the $C^0$ case, such a result was proved already by Andersson and Driver \cite{anderssondriver} and was later generalized to the case of vector-valued Laplacians by B\"ar and Pf\"affle \cite{bpapproximation}.

\medskip

In this paper, will further generalize these results to the case that $M$ is a compact Riemannian manifold with boundary. It turns out that in this case, the space $H_{x;\tau}(M)$ has to be replaced by the space $H_{x;\tau}^{\mathrm{refl}}(M)$ of piecewise {\em reflected} geodesics. Roughly speaking, a reflected geodesic is a path $\gamma: [0, t] \longrightarrow M$ that is a geodesic near all times $s \in [0, t]$ such that $\gamma(s) \notin \partial M$ and at the times $s \in [0, t]$ with $\gamma(s) \in \partial M$, $\gamma$ reflects with the angle of reflection equal to the angle of incidence. 

We will prove a version of the approximation formula \eqref{FiniteDimensionalApproximation} for general self-adjoint Laplace type operators, acting on sections of a metric vector bundle over $M$, and for various boundary conditions. The class of boundary conditions we consider, will be called {\em involutive boundary conditions}; it is a subclass of mixed boundary conditions and includes standard Dirichlet and Neumann boundary conditions, as well as absolute and relative boundary conditions on differential forms and vector fields.

For the operator $\Delta + V$ on functions with Neumann boundary conditions, formula \eqref{FiniteDimensionalApproximation} remains valid as it is except that one has to replace the integration domain $H_{x;\tau}(M)$ by $H_{x;\tau}^{\mathrm{refl}}(M)$. In the case of Dirichlet boundary conditions, \eqref{FiniteDimensionalApproximation} becomes
\begin{equation*}  
u(t, x) = \lim_{|\tau|\rightarrow 0} \fint_{H_{x;\tau}(M)} \exp \left( - \frac{1}{4}\int_0^t \bigl|\dot{\gamma}(s)\bigr|^2 \dd s - \int_0^t V\bigl(\gamma(s)\bigr) \dd s \right) u_0\bigl(\gamma(t)\bigr) (-1)^{\mathrm{refl}(\gamma)} \dd\gamma,
\end{equation*}
where $\mathrm{refl}(\gamma)$ denotes the number of boundary reflections of the path $\gamma$.

Our proof uses a well-known theorem of Chernoff about proper families (see Prop.~\ref{PropChernoff} below). In order to show that this result can be applied, we make a time rescaling and reformulate our path integral formula using the broken billiard flow (which is defined in Section~\ref{SectionReflectedGeodesics}). Of course, our proof also works in the case that $\partial M = \emptyset$. This gives a new proof of the path integral formula \eqref{FiniteDimensionalApproximation} in the closed case, which we believe is simpler than the ones existing in the literature: It neither uses stochastic analysis nor knowledge of the short-time asymptotics of the heat kernel, only basic properties of the geodesic flow (which is the analog of the broken billiard flow in the closed case). The result that the approximation \eqref{FiniteDimensionalApproximation} is valid in $L^p$ for $u_0 \in L^p$ also seems to be new.

\medskip

The paper is structured as follows. To set notation, we first review some basic results on vector-valued Laplace type operators acting on vector bundles, and we introduce the class of boundary conditions considered in this paper. Afterwards, we introduce the broken billiard flow on a manifold with boundary, reflected geodesics and the space $H_{x;\tau}^{\mathrm{refl}}(M)$. In the final section, we formulate and prove our results on time-slicing path integrals.

\medskip

\textbf{Acknowledgements}. I would like to thank Christian Bär, Rafe Mazzeo, Franziska Beitz and Florian Hanisch for many helpful discussions. Furthermore, I am indepted to Potsdam Graduate School, The Fulbright Program and SFB 647 for financial support.

\section{Involutive Boundary Conditions and the Heat Equation} \label{SectionBoundaryConditionsHeatEquation}

Let $L$ be a formally self-adjoint Laplace type operator in the sense of \cite{bgv}, acting on sections of a metric vector bundle $\V$ over a compact $n$-dimensional Riemannian manifold $M$, possibly with boundary (here a metric vector bundle means a real or complex vector bundle with a positive definite scalar product or Hermitean form, respectively). For any such operator $L$, there exists a unique metric connection $\nabla$ and a unique symmetric endomorphism field $V$ such that
\begin{equation} \label{LaplaceDecomposition}
  L = \nabla^* \nabla + V,
\end{equation}
where $\nabla^*$ is the $L^2$-adjoint of the differental operator $\nabla$. 

This determines a {\em path-ordered exponential} $\mathcal{P}(\gamma)$ along piecewise smooth paths $\gamma: [a, b] \longrightarrow M$, needed subsequently: Let $P(s) \in \mathrm{Hom}(\V_{\gamma(a)}, \V_{\gamma(s)}\bigr)$ be the unique solution to the ordinary differential equation
\begin{equation} \label{ODEforE}
  \frac{\nabla}{\dd t} P(s) = V\bigl(\gamma(s)\bigr) P(s), ~~~~~~~~~~ P(a) = \id,
\end{equation}
where $\nabla$ and $V$ are the connection and potential determined by \eqref{LaplaceDecomposition}. The path-ordered exponential $\mathcal{P}(\gamma)$ is then defined by $\mathcal{P}(\gamma) := P(b) \in \mathrm{Hom}(\V_{\gamma(a)}, \V_{\gamma(b)})$.

For example, if $V \equiv 0$ along $\gamma$, we have $\mathcal{P}(\gamma) = [\gamma\|_0^t]$,
the parallel transport map along $\gamma$ with respect to $\nabla$. In the scalar case, when $\nabla = d + i\omega$ for some one-form $\omega \in \Omega^1(M)$, the differential equation \eqref{ODEforE} can be solved explicitly, giving
\begin{equation}\label{ExplicitFormulaE}
  \mathcal{P}(\gamma) = \exp \left( - i \int_0^t \omega|_{\gamma(s)} \cdot \dot{\gamma}(s) \,\dd s + \int_0^t V\bigl(\gamma(s)\bigr)\,\dd s \right).
\end{equation}
In the general vector-valued case, however, there is usually no closed-form solution for $\mathcal{P}( \gamma)$.

\begin{remark}[Invertibility]
 $\mathcal{P}(\gamma)$ is always invertible, and $\mathcal{P}(\gamma)^{-1} = Q(t)$, where $Q(s) = P(s)^{-1} \in \mathrm{Hom}(\V_{\gamma(s)}, \V_{\gamma(a)})$ satisfies the differential equation
\begin{equation} \label{ODEforEinverse}
  \nabla_s Q(s) = -Q(s)V\bigl(\gamma(s)\bigr), ~~~~~~~~ Q(a) = \id,
\end{equation}
as is easy to verify by differentiating the identity $\id = P(s)^{-1}P(s)$ and using uniqueness of solutions.
\end{remark}

\begin{remark}[Multiplicativity] \label{RemarkPBMultiplicativity}
$\mathcal{P}(\gamma)$ is multiplicative, in the sense that if $\gamma_1$, $\gamma_2$ are paths parametrized by $[0, t_1]$ and $[0, t_2]$ respectively, such that $\gamma_1(t) = \gamma_2(0)$, then we have $\mathcal{P}(\gamma_2)\mathcal{P}(\gamma_1) = \mathcal{P}(\gamma_1 * \gamma_2)$, where 
\begin{equation} \label{ConcatentationDefinition}
  (\gamma_1 * \gamma_2)(s) :=\begin{cases} \gamma_1(s) &\text{if}~s \leq t_1\\ \gamma_2(s-t_1) &\text{if}~t_1 \leq s \leq t_1+t_2\end{cases}
\end{equation}
denotes the concatenation. Again, this is easy to verify using uniqueness of solutions for ordinary differential equations.
\end{remark}

Given a formally self-adjoint Laplace-type operator and a section $u_0 \in C^\infty(M, \V)$, we can consider the heat equation
\begin{equation} \label{TheHeatEquation}
  \left(\frac{\partial}{\partial t}  + L \right) u(t, x) = 0, ~~~~~~ u(0, x) = u_0(x)
\end{equation}
for time-dependent sections $u$ of $\V$. In case that $M$ has a boundary, one has to require boundary conditions to make the heat equation well-posed. 

\begin{definition}[Involutive Boundary Conditions] \label{DefReflectingBoundaryConditions}
Let $L$ be a formally self-adjoint Laplace type operator, acting on sections of a metric vector bundle $\V$ over a compact Riemannian manifold with boundary $M$. A symmetric endomorphism field $B \in C^\infty(\partial M, \mathrm{End}(\V|_{\partial M}))$ is called an {\em involutive boundary operator for} $L$ if $B^2 = \id$ and if $B$ is covariantly constant with respect to the connection determined by \eqref{LaplaceDecomposition}. To such an boundary operator $B$, there corresponds a splitting
\begin{equation} \label{BoundarySplitting}
\V|_{\partial M} = \W^+ \oplus \W^-
\end{equation}
into the eigenspaces of the eigenvalues $\pm 1$ (notice that only these two eigenvalues are possible since $B^2=\id$). We say that a section $u \in C^\infty(M, \V)$ {\em satisfies the boundary condition} defined by $B$ if
\begin{equation} \label{ReflectingBoundaryCondition}
  \nabla_\n u|_{\partial M} \in C^\infty(\partial M, \mathcal{W}^-), ~~~~~~~ u|_{\partial M} \in C^\infty(\partial M, \mathcal{W}^+),
\end{equation}
where $\n \in C^\infty(\partial M, N\partial M)$ denotes the interior normal vector to the boundary.
\end{definition}

\begin{notation}
For a boundary operator $B$, let $C^\infty_B(M, \V)$ be the space of smooth sections of $\V$ that satisfy the boundary condition and let $H^2_B(M, \V) := \overline{C^\infty_B(M, \V)} \subseteq H^2(M, \V)$ be its closure with respect to some $H^2$ norm. 
\end{notation}

\begin{remark} \label{remarkOnMixedBC}
The class of involutive boundary conditions is closely related to the class of {\em mixed boundary conditions}, as defined e.g.\ in \cite[Section~1.5.3]{GilkeySpectralGeometry}. However, mixed boundary conditions are slightly more general, therefore we stick to the term "involutive boundary condition" in this paper (c.f.\ Chapter~II of \cite{GreinerHeatEquation}, Section~1.11.2 in \cite{gilkey95} or Sections~1.4-1.6 in \cite{GilkeySpectralGeometry} for a much more general discussion). 
\end{remark}

Involutive boundary conditions ensure that the operator $L$ as an unbounded operator on $L^2(M, \V)$ is essentially self-adjoint on $C^\infty_B(M, \V)$ and self-adjoint on $H^2_B(M, \V)$. We say that $L$ {\em is endowed with involutive boundary condition} $B$ if $L$ has the latter domain and $B$ is an involutive boundary operator.

It is clear that any involutive boundary condition is in particular elliptic, hence it follows from the standard theory that when $L$ is endowed with such a boundary condition, it has a discrete spectrum $\lambda_1 \leq \lambda_2 \leq \dots \rightarrow \infty$, where the eigenvalues have finite multiplicity and the corresponding eigenfunctions $\phi_j$ are contained in $C^\infty_B(M, \V)$ (see e.g.\ the references in Remark~\ref{remarkOnMixedBC} above).
In particular, $L$ generates a strongly continuous semigroup $e^{-tL}$, defined by spectral calculus. For any initial condition $u_0 \in L^2(M, \V)$, the function
\begin{equation*}
u(t, x) := (e^{-tL}u_0)(x)
\end{equation*}
satisfies the heat equation \eqref{TheHeatEquation} with initial condition $u(0, x) = u_0(x)$.

\medskip

We now give a couple of examples for involutive boundary conditions.

\begin{example}[Dirichlet and Neumann] \label{ExampleBoundaryConditionsDirichletNeumann}
For any Laplace type operator, there is the {\em Dirichlet} boundary condition $u|_{\partial M} = 0$, associated to the boundary operator $B = -\id$, and the {\em Neumann} boundary condition $\nabla_\n u|_{\partial M} = 0$ associated to the boundary operator $B= \id$. Here $\id$ denotes the identity endomorphism field of $\V|_{\partial M}$, which is parallel with respect to any connection on $\V$ (or more precisely: with respect to any connection on the bundle $\mathrm{End}(\V)$ induced from a connection on $\V$). Both are therefore involutive boundary conditions.
\end{example}

\begin{nonexample}[Robin Boundary Conditions]
Given a metric connection $\nabla$, the {\em generalized Neumann} boundary conditions or {\em Robin} boundary conditions 
\begin{equation} \label{RobinBoundaryConditions}
  \nabla_\n u|_{\partial M} + A u|_{\partial M} = 0.
\end{equation}
for an endomorphism field $A \in C^\infty(\partial M, \End(\V))$ are {\em not} involutive boundary conditions for operators of the form $L = \nabla^*\nabla + V$, unless $A \equiv 0$.
\end{nonexample}

\begin{example}[Boundary Conditions on Differential Forms] \label{ExampleBoundaryConditionsOnForms}
Let $\V = \Lambda^k T^*M$ be the bundle of $k$-forms.
Any $\omega \in \Lambda^k T^*M$ can be decomposed at the boundary as 
\begin{equation*}
\omega = \omega_0 + dr \wedge \omega_1,~~~~~~~~\omega_0 \in \Lambda^kT^*\partial M, ~~\omega_1 \in \Lambda^{k-1}T^*\partial M,
\end{equation*}
where $dr := \n^\flat$. Hence for the exterior products of the cotangent bundle, we have the orthogonal splitting
\begin{equation*}
  \Lambda^k T^*M|_{\partial M} \cong \Lambda^k T^*\partial M \oplus dr \wedge \Lambda^{k-1}T^*\partial M.
\end{equation*}
Defining $B$ to be equal to $1$ on one of these factors and equal to $-1$ on the other will induce involutive boundary conditions for Laplace type operators $L = \nabla^*\nabla+V$ on $\V$, where $\nabla$ is any metric connection on $\V$. Specifically, setting
\begin{equation}
  \mathcal{W}^+ := \Lambda^k T^*\partial M, ~~~~~~ \mathcal{W}^- := dr \wedge \Lambda^{k-1}T^*\partial M
\end{equation}
gives the so-called {\em absolute boundary conditions}. 
Setting 
\begin{equation} \label{RelativeBoundaryCondition1}
  \mathcal{W}^+ := dr \wedge \Lambda^{k-1}T^*\partial M, ~~~~~~ \mathcal{W}^- := \Lambda^k T^*\partial M
\end{equation}
gives {\em relative boundary conditions}. 
\end{example}

The examples show that the class of involutive boundary conditions includes most standard types of boundary conditions. Let us make a warning here that "involutive" is not standard terminology, but such a class of boundary conditions doesn't seem to have a name in the literature yet.

\subsection{Reflected Geodesics and the Broken Billiard Flow} \label{SectionReflectedGeodesics}

Let $M$ be a compact $n$-dimensional Riemannian manifold with boundary. Denote by $\n \in C^\infty(\partial M, N\partial M)$ the interior unit normal field. We say that a vector $v \in TM|_{\partial M}$ {\em points inward} if $\<v, \n\> > 0$ and we say that it {\em points outward} if $\<v, \n\> < 0$. If $v$ points neither outward nor inward, then clearly $v \in T\partial M$.

\begin{notation}[Reflection at the Boundary]
Set
\begin{equation}
Rv := v - 2\<v, \n\>\n, ~~~~~~ v \in TM|_{\partial M}
\end{equation}
for the reflection at $T\partial M$. We have $R \in C^\infty(\partial M, \End(TM|_{\partial M}))$.
\end{notation}

\begin{definition}[Reflected Geodesics] \label{DefReflectedGeodesic}
A {\em reflected geodesic} is a continuous map $\gamma:[a, b] \longrightarrow M$ such that
\begin{enumerate}[(i)]
    \item $\gamma$ hits the boundary only at finitely many times $a \leq \sigma_1 < \sigma_2 < \dots < \sigma_k \leq b$, $k \in \N_0$;
    \item on each of the intervals $(a, \sigma_1)$, $(\sigma_1, \sigma_2)$, $\dots$, $(\sigma_{k-1}, \sigma_k)$, $(\sigma_k, b)$, $\gamma$ is a geodesic;
    \item $\dot{\gamma}(\sigma_j\pm)\notin T\partial M$, where $\dot{\gamma}(\sigma_j\pm)$ denotes the right-/left-sided derivative, i.e.\ $\gamma$ always hits the boundary transversally;
\item we have $\dot{\gamma}(\sigma_j-) = R\dot{\gamma}(\sigma_j+)$ for each $j=1, \dots, k$, that is, $\gamma$ reflects with the angle of reflection equal to the angle of incidence. (If $\sigma_1=a$ or $\sigma_k = b$, this condition is empty for $j=1$ respectively $j=k$.)
  \end{enumerate}
\end{definition}

The requirement (iii) excludes geodesics that "scratch along the boundary", the so-called grazing rays, which can appear e.g.\ when $M$ is the exterior of a ball in $\R^n$. 

% T underhalbstetig $\lim_{v \rightarrow v_0} T(v) \geq T(v_0)$ impliziert: Omega_t is offen.
\begin{notation} 
For $v \in TM$, let $T(v)$ be the supremum over all times $t>0$ such that a reflected geodesic $\gamma_v:[0, t] \longrightarrow M$ exists with $\dot{\gamma}_v(0+) = v$, or if $v \in TM|_{\partial M}$ is pointing outward, with $\dot{\gamma}_v(0+) = Rv$). Denote 
\begin{equation*}
  \Omega_t := \bigl\{ v \mid  T(v) > t \bigr\}
\end{equation*}
for the set of vectors $v$ such that there exists a reflected geodesic with initial condition $v$ (respectively $Rv$) up to a time larger than $t$.
\end{notation}

Obviously, we have $T(v) = -\infty$ for $v \in T\partial M$ and $T(v) > 0$ otherwise. Hence $\Omega_0 = TM \setminus T\partial M$. Since restrictions of reflected geodesics are reflected geodesics, we have furthermore $\Omega_t \supseteq \Omega_{t^{\prime}}$ for $t \leq t^\prime$. 

\begin{remark}
We generally do {\em not} have the equality $\Omega_t = TM \setminus T\partial M$ here, as two things could go wrong:
\begin{enumerate}[(a)]
  \item We may have $\lim_{s \rightarrow t_0} \dot{\gamma}_v(s) \in T\partial M$.
  \item There may be infinitely many reflections in finite time, i.e.\ reflection times $\sigma_1 < \sigma_2 < \dots$ converging to a time $t_0 < \infty$ as $j\rightarrow \infty$.
\end{enumerate}
In both cases, one cannot continue $\gamma_v$ beyond the time $t_0$ (at least not as a reflected geodesic in the sense of Def.~\ref{DefReflectedGeodesic}). In case (a), the "physically reasonable" outcome would be that $\gamma_v$ "glides along the boundary" for $t>T$, but this would mean that $\gamma_v$ is a geodesic in $\partial M$, not in $M$ ($\nabla_s \dot{\gamma}(s)$ would be proportional to $-\n$).

If $M$ is convex (i.e.\ the second fundamental form of the boundary points outward everywhere), then (a) cannot happen. Also (b) cannot happen in the case that $\partial M$ is smooth (which is always assumed here) and convex, at least if $M$ is a subset of $\R^2$, but there is an example of a convex $M \subset \R^2$ with only $C^2$ boundary, where (b) can occur \cite{Halpern}. However, to the author's knowledge, there is no (non-convex) example of a manifold $M$ with {\em smooth} boundary in literature, where (b) happens. The author does not know if (b) can happen at all. 
\end{remark}

\begin{lemma} \label{LemmaOpenAndDense}
For each $t\geq 0$, the set $\Omega_t$ is an open set of full measure in $TM$ and for each $x \in M$, the set $\Omega_{t, x} := \Omega_t \cap T_xM$ is an open set of of full measure in $T_x M$.
\end{lemma}

\begin{proof}
That the sets $\Omega_{t, x}$ and $\Omega_t$ have full measure is a result from the theory of dynamical systems and ergodic theory, see for example Chapter 6 of \cite{SinaiFominKornfeld}. Furthermore, that the sets $\Omega_{t, x}$ and $\Omega_t$ are open is due to the fact that solutions of ordinary differential equations depend continuously on the initial data. More precisely, one can show by induction on the number of reflections that for each $v \in \Omega_t$, there exists a small neighborhood of $v$ such that for each $w$ in that neighborhood, there exists a reflected geodesic $\gamma_w$ up to time larger than $t$, and the value $\dot{\gamma}_w(t)$ depends continuously on $w$ in this neighborhood.
\end{proof}

\begin{definition}[Broken Billiard Flow]
The {\em broken billiard flow} is the measurable map $\Theta: \R \times TM \longrightarrow TM$ defined as follows. Set $\Theta_0(v) := v$. For $t>0$ and $v \in \Omega_t$, we set
\begin{equation*}
  \Theta_t(v) = \dot{\gamma}_v\bigl(t\bigr)
\end{equation*}
where $\gamma_v:[0, t] \longrightarrow M$ is the reflected geodesic with $\dot{\gamma}_v(0+) = v$, respectively $\dot{\gamma}_v(0+) = Rv$ if $v \in TM|_{\partial M}$ is outward directed. For $v \notin \Omega_t$, set $\Theta_t(v) = v$. For negative times, $t<0$, set $\Theta_t(v) := -\Theta_{-t}(-v)$.
\end{definition}

\begin{remark}
If $\partial M = \emptyset$, this is just the usual geodesic flow on the tangent bundle.
\end{remark}

\begin{remark}
The broken billiard flow is often considered on the unit sphere bundle $SM$ instead of on $TM$. Because we have
\begin{equation} \label{RescalingPropertyFlow}
  \Theta_t(v) = |v|\Theta_{t|v|} (v/|v|),
\end{equation}
both flows can be obtained from one another.
\end{remark}

Because $TM \setminus \Omega_t$ is a zero set, for each $t \in \R$, the broken billiard map $\Theta_t$ is almost invertible, in the sense that $\Theta_t\circ\Theta_{-t} = \id$ except for a zero set. Furthermore, it is well known \cite[Lemma~4]{SinaiFominKornfeld} that $\Theta_t$ preserves the volume of $TM$, just as the geodesic flow does on a complete Riemannian manifold without boundary.

\subsection{Reflected Path Spaces} \label{SectionReflectedPathSpaces}

\begin{notation}
Let $\tau = \{ 0=\tau_0 < \tau_1 < \dots < \tau_N = t\}$ be a partition of the interval $[0, t]$ and let $x \in M$. In the case that $x \notin \partial M$, set
\begin{equation*}
\begin{aligned}
  H_{x;\tau}^{\mathrm{refl}}(M) := \bigl\{ \gamma \in C^0([0, t], M) \mid \gamma(0)=x, \gamma|_{[\tau_{j-1}, \tau_j]}~&\text{is a reflected geodesic}, \gamma(\tau_j)\notin \partial M\bigr\}.
\end{aligned}
\end{equation*}
In the case that $x \in \partial M$, we let $H_{x;\tau}^{\mathrm{refl}}(M)$ be defined exactly the same, except that we additionally take the product with $\Z_2$ (here we always use the multiplicative representation $\Z_2 = \{+1, -1\}$).
\end{notation}

If $x \notin \partial M$, then $H_{x;\tau}^{\mathrm{refl}}(M)$ is just the space of piecewise reflected geodesics starting at $x$. On the other hand, if $x \in \partial M$, then 
\begin{equation*}
H_{x;\tau}^{\mathrm{refl}}(M)=\{\text{piecewise reflected geodesics starting at}~x\} \times \Z_2,
\end{equation*}
 i.e.\ the paths carry the additional information of a number $\epsilon \in \Z_2$. Heuristically, this number encodes whether or not the path reflects at time zero, i.e.\ whether it "starts inward or it starts outward and reflects immediately". This number $\epsilon$ will be called the {\em sign} of the path.

We will often just write $\gamma$ instead of $(\gamma, \epsilon)$ for elements of $H_{x;\tau}^{\mathrm{refl}}(M)$, $x \in \partial M$ (especially when integrating over this space) and consider $\gamma$ as an ordinary path "with decoration". However, the additional information on the sign has to be kept in mind.

\medskip

Before we describe the manifold structure on $H_{x;\tau}^{\mathrm{refl}}(M)$, we introduce the $B$-path-ordered exponential along elements $\gamma \in H_{x;\tau}(M)$, where $B$ is an involutive boundary condition for a Laplace-type operator $L= \nabla^*\nabla + V$, acting on a vector bundle $\V$ over $M$.

\begin{definition}[$B$-path-ordered Exponential] \label{DefBPathOrderedExponential}
Let $L = \nabla^* \nabla + V$ be a self-adjoint Laplace type operator with involutive boundary condition $B$. The {\em $B$-path-ordered exponential} $\mathcal{P}_B(\gamma)$ along paths $\gamma \in H_{x;\tau}^{\mathrm{refl}}(M)$ is defined as follows. Let $\sigma_1 < \dots < \sigma_k$ be the times in $(0, t)$ such that $\gamma(\sigma_j) \in \partial M$ (i.e.\ $\gamma(s) \notin \partial M$ for $s \neq 0$, $s\neq \sigma_j$, $j=1, \dots, k$). Set 
\begin{equation} \label{BPathOrderedExponential}
  \mathcal{P}_B(\gamma) := \mathcal{P}(\gamma|_{[\sigma_k, t]})\, B \,\mathcal{P}(\gamma|_{[\sigma_{k-1}, \sigma_k]})\,B \cdots B\, \mathcal{P}(\gamma|_{[\sigma_1, \sigma_2]}) \,B\, \mathcal{P}(\gamma|_{[0, \sigma_1]}) A,
\end{equation}
where $\mathcal{P}(\gamma|_{[\sigma_{j-1}, \sigma_j]})$ is defined as in the beginning of Section~\ref{SectionBoundaryConditionsHeatEquation}, and $A := \id$ if $x \notin \partial M$ or if $x \in \partial M$ and the sign of $\gamma$ is $+1$, while $A := B$ if $x \in \partial M$ and the sign of $\gamma$ is $-1$.
That is, we take the usual path-ordered exponential, but whenever the path $\gamma$ hits the boundary, we use the boundary involution $B$ before continuing to solve the differential equation \eqref{ODEforE}.
In particular, if $V=0$, we obtain the {\em $B$-parallel transport}, denoted by $[\gamma\|_0^t]_B$.
\end{definition}

Directly from the properties of $\mathcal{P}(\gamma)$ follows that also $\mathcal{P}_B(\gamma)$ is invertible for any path $\gamma$, and that $\mathcal{P}_B(\gamma)$ is multiplicative, i.e.\
\begin{equation} \label{MultiplicativityPB}
\mathcal{P}(\gamma_2)\mathcal{P}(\gamma_1) = \mathcal{P}(\gamma_1 * \gamma_2)
\end{equation}
for suitable paths $\gamma_1, \gamma_2$.

The $B$-path-ordered exponential can be used to obtain a manifold structure on $H_{x;\tau}^{\mathrm{refl}}(M)$. Notice that on the vector bundle $\V := TM$, there is a natural boundary operator, namely $B:=R$, the reflection at $T\partial M$. We define the {\em reflected anti-development map}
\begin{equation*}
  U_R(\gamma)(s) := \int_0^s[\gamma\|_0^u]_R^{-1} \dot{\gamma}(u) \dd u.
\end{equation*}
Then $U_R$ maps $H_{x;\tau}^{\mathrm{refl}}(M)$ to $H_{0; \tau}(T_xM)$, the vector space of piece-wise polygon paths starting at zero in $T_xM$. To verify this, we need to show that $U_R(\gamma)$ is a straight line on each of the intervals $[\tau_{j-1}, \tau_j]$, $j=1, \dots, N$. This is clear for all times $s$ where $\gamma(s) \notin \partial M$, since differentiating $U_R(\gamma)(s)$ twice with respect to $s$ gives zero. If now $\gamma(s) \in \partial M$, then
\begin{equation*}
  \frac{\dd}{\dd s} U_R(\gamma)(s-) = [\gamma\|_0^{s-}]_R^{-1} \dot{\gamma}(s-) = (R[\gamma\|_0^{s+}]_R)^{-1} R\dot{\gamma}(s+) =\frac{\dd}{\dd s} U_R(\gamma)(s+),
\end{equation*}
because the two reflections cancel each other. Hence $\gamma$ does not have a kink at $s$ and is therefore a straight line near $s$. Furthermore, $U_R$ is injective. This is clear for $x \notin \partial M$. If $x \in \partial M$, then each piece-wise reflected geodesic $\gamma$ starting at $x$ appears twice, once with negative sign and one with positive sign. But by definition of the reflected anti-development, we have $U_R(\gamma, +1)(s) = R U_R(\gamma, -1)(s)$. 

Because of Lemma~\ref{LemmaOpenAndDense}, the image $U_R(H_{x;\tau}^{\mathrm{refl}}(M)) \subseteq H_{0, \tau}(T_xM)$ is an open and dense set of full measure, so that one obtains a manifold structure on $H_{x;\tau}^{\mathrm{refl}}(M)$ by using $U_R$ as global chart. 

\begin{remark}
If $\partial M = \emptyset$, then $\Omega_t = TM$ for all $t$, and we have $H_{x;\tau}^{\mathrm{refl}}(M) = H_{x;\tau}(M)$ as defined by \eqref{DefinitionHTau}. 
\end{remark}

Notice that for two partitions $\tau$ and $\tau^\prime$ of intervals $[0, t]$ and $[0, t^\prime]$ respectively, if $\gamma \in H_{x;\tau}^{\mathrm{refl}}(M)$ and $\gamma^\prime \in H_{\gamma(t);\tau^\prime}^{\mathrm{refl}}(M)$, then the concatenation $\gamma * \gamma^\prime$ (as defined in \eqref{ConcatentationDefinition}) is contained in $H_{x; \tau * \tau^\prime}^{\mathrm{refl}}(M)$. This fact is used in the following Lemma.

\begin{lemma}[A Co-Area Formula] \label{LemmaConcatenation}
   Let $\tau = \{0 = \tau_0 < \tau_1 < \dots < \tau_N = t\}$ and $\tau^\prime = \{0 = \tau_0 ^\prime< \tau_{1}^\prime < \dots < \tau_{N^\prime}^\prime = t^\prime\}$ be partitions of the interval $[0, t]$ and $[0, t^\prime]$. Then for any integrable function $F$ on $H_{x; \tau*\tau^\prime}^{\mathrm{refl}}(M)$, we have the co-area formula
   \begin{equation*}
     \int_{H_{x; \tau*\tau^\prime}^{\mathrm{refl}}(M)} F(\gamma)\, \dd^{\Sigma\text{-}H^1} \gamma = \int_{H_{x; \tau}^{\mathrm{refl}}(M)}\int_{H_{\gamma(t); \tau^\prime}^{\mathrm{refl}}(M)} F(\gamma * \gamma^\prime) \,\dd^{\Sigma\text{-}H^1} \gamma^\prime \,\dd^{\Sigma\text{-}H^1} \gamma,
   \end{equation*}
   where each of the spaces carries the discrete $H^1$-metric defined in \eqref{DiscretizedMetric}.
\end{lemma}

\begin{proof}
 Consider the restriction maps
 \begin{equation*}
   \mathrm{res}: H_{x; \tau * \tau^\prime}^{\mathrm{refl}}(M) \longrightarrow H_{x;\tau}^{\mathrm{refl}}(M), ~~~~~ \gamma \longmapsto \gamma|_{[0, t]}
 \end{equation*}
 We show that $\mathrm{res}$ is a Riemannian submersion, i.e.\ that for any $\gamma * \gamma^\prime \in H_{x;\tau * \tau^\prime}^{\mathrm{refl}}(M)$, the linear map
\begin{equation*}
 d\mathrm{res}|_{\gamma * \gamma^\prime} : T_{\gamma*\gamma^\prime} H_{x;\tau * \tau^\prime}^{\mathrm{refl}}(M) \longrightarrow T_\gamma H_{x;\tau}^{\mathrm{refl}}(M)
\end{equation*}
is an isometry when restricted to the orthogonal complement of its kernel. The kernel of $d\mathrm{res}|_{\gamma*\gamma^\prime}$ is the set of Jacobi fields that are zero up to time $t$. Therefore, looking at the formula \eqref{DiscretizedMetric} for the metric, the orthogonal complement of the kernel is the set of Jacobi fields $X$ such that
\begin{equation*}
 \nabla_s X\bigl((\tau*\tau^\prime)_{j-1}+\bigr) = 0, ~~~~~ j = N+1, \dots, N+N^\prime.
\end{equation*}
Hence, if $X$ is such a vector field in the orthogonal complement, then
\begin{equation*}
  \|X\|_{\Sigma\text{-}H^1}^2 = \sum_{j=1}^{N+N^\prime} \bigl| \nabla_s X\bigl((\tau*\tau^\prime)_{j-1}+\bigr)\bigr|^2 \Delta_j\tau = \sum_{j=1}^{N} \bigl| \nabla_s X\bigl(\tau_{j-1}+\bigr)\bigr|^2 \Delta_j\tau = \bigl\|X|_{[0, t]}\bigr\|_{\Sigma\text{-}H^1}^2
\end{equation*}
so that because $d\mathrm{res} \,X = X|_{[0, t]}$, $d\mathrm{res}$ is indeed an isometry when restricted to this subspace.
From the co-area formula \cite[Thm.~III..2]{ChavelRiemannian}, we obtain
\begin{equation*}
  \int_{H_{x; \tau*\tau^\prime}^{\mathrm{refl}}(M)} F(\gamma)\, \dd^{\Sigma\text{-}H^1} \gamma = \int_{H_{x; \tau}^{\mathrm{refl}}(M)}\int_{\res^{-1}(\gamma)} F(\eta) \,\dd^{\Sigma\text{-}H^1} \eta \,\dd^{\Sigma\text{-}H^1} \gamma,
\end{equation*}
so the proof is finished if we show that the map $\mathrm{ext}_\gamma$ given by
\begin{equation*}
\mathrm{ext}_\gamma: H_{\gamma(t);\tau^\prime}(M) \longrightarrow \mathrm{res}^{-1}(\gamma), ~~~~~~~ \gamma^\prime \mapsto \gamma * \gamma^\prime
\end{equation*}
is an isometry. So let $X \in T_{\gamma^\prime} H_{\gamma(t);\tau^\prime}(M)$. Then 
\begin{equation*}
\bigl(d \mathrm{ext}_\gamma|_{\gamma^\prime} X\bigr)(s) = \begin{cases} 0 & 0 \leq s \leq t \\ X(s-t) & t < s \leq t+ t^\prime \end{cases}
\end{equation*}
which implies $\|d \mathrm{ext}_\gamma|_{\gamma^\prime} X\|_{\Sigma\text{-}H^1} = \|X\|_{\Sigma\text{-}H^1}$. Thus $\mathrm{ext}_\gamma$ is indeed an isometry for every $\gamma$ and the lemma follows.
\end{proof}

\subsection{Reflecting Path Integrals} \label{SectionReflectingPathIntegrals}

We can now give a path integral formula for the heat operator in the case that $M$ is a compact manifold with boundary.

\begin{theorem}[The Heat Operator as a Reflecting Path Integral] \label{ThmPathIntegralNonSmooth}
~~Let $L$ be a self-adjoint Laplace type operator, acting on sections of a metric vector bundle $\V$ over a compact Riemannian manifold $M$ with boundary, endowed with involutive boundary condition $B$. Let $\mathcal{P}_B(\gamma)$ denote the $B$-path-ordered exponential, induced by $L$ as in Def.~\ref{DefBPathOrderedExponential}. For a partition $\tau = \{0 = \tau_0 < \tau_1 < \dots < \tau_N = t\}$, define
\begin{equation} \label{DefPtau}
  P_\tau u(x) := \fint_{H_{x;\tau}^{\mathrm{refl}}(M)} \exp\left(-\frac{1}{4} \int_0^t \bigl|\dot{\gamma}(s)\bigr|^2 \dd s\right)  \mathcal{P}_B(\gamma)^{-1} u\bigl(\gamma(t)\bigr) \, \dd\gamma,
\end{equation}
where the slash over the integral sign denotes divison by $(4\pi)^{\mathrm{dim}(H_{x;\tau}(M))/2}$. Then
\begin{equation} \label{PathIntegralFormulaNonSmooth}
  e^{-tL}u = \lim_{|\tau|\rightarrow 0} P_\tau u,
\end{equation}
where the limit goes over any sequence of partitions the mesh of which tends to zero and the section $u$ is in any of the spaces $C^0(M, \V)$ or $L^p(M, \V)$, $1 \leq p < \infty$ (with convergence in the respective space).
\end{theorem}

\begin{remark}
  Of course, the definition \eqref{DefPtau} makes sense pointwise only for $u \in C^0(M, \V)$. However, we will show that each operator $P_\tau$ is bounded with respect to the $L^p$ norm, $1 \leq p < \infty$, so that it extends uniquely to a bounded operator on $L^p(M, \V)$ (also denoted by $P_\tau$), because $C^0(M, \V)$ is dense in $L^p(M, \V)$. For a general $u \in L^p(M, \V)$, $P_\tau u$ is defined by formula \eqref{DefPtau} almost everywhere.
\end{remark}

\begin{example}[The Laplace-Beltrami Operator]
If we have $L = \Delta + V$ for a potential $V$ with Dirichlet boundary conditions (i.e.\ $B\equiv -1$), then we have $\mathcal{P}_B(\gamma) = (-1)^{\mathrm{refl}(\gamma)}$, where $\mathrm{refl}(\gamma)$ denotes the number of reflections, i.e.\ the number of times $0 \leq s \leq t$ such that $\gamma(s) \in \partial M$. We therefore arrive at the formula
\begin{equation*}  
u(t, x) = \lim_{|\tau|\rightarrow 0} \fint_{H_{x;\tau}(M)} \exp \left( - \frac{1}{4}\int_0^t \bigl|\dot{\gamma}(s)\bigr|^2 \dd s - \int_0^t V\bigl(\gamma(s)\bigr) \dd s \right) u_0\bigl(\gamma(t)\bigr) (-1)^{\mathrm{refl}(\gamma)} \dd\gamma
\end{equation*}
from the introduction.
If we consider the Neumann boundary conditions, then $B=1$ and $\mathcal{P}_B(\gamma) \equiv 1$, so the factor $(-1)^{\mathrm{refl}(\gamma)}$ has to be replaced by one.
\end{example}

The proof of Thm.~\ref{ThmPathIntegralNonSmooth} is based on the following result, which is due to Chernoff \cite{chernoff}. In the following form, the it can be found in \cite[Prop.~1]{WeizsaeckerSmolyanov05} and \cite[Thm.~2.8]{bpapproximation}, where it was already used to approximate the heat semigroup on closed manifolds.

\begin{proposition}[Chernoff] \label{PropChernoff}
Let $(P_t)_{t\geq 0}$ be a family of bounded linear operators on a Banach space $E$ and assume that $P_t$ is a {\em proper family}, i.e.\ 
\begin{enumerate}[(i)]
\item $\|P_t\| = 1 + O(t)$ as $t \rightarrow 0$;
\item $P_t$ is strongly continuous with $P_0 = \id$;
\item $P_t$ has an infinitesimal generator, meaning that there exists a (possibly unbounded) closed operator $L$ on $E$ with dense domain $\mathrm{dom}(L)$ that generates a strongly continuous semigroup $e^{-tL}$ and such that
\begin{equation*}
  \frac{1}{t} \bigl( P_t u - u\bigr) \longrightarrow -Lu
\end{equation*}
as $t \rightarrow 0$ for all $u \in E$ of the form $u = e^{-\varepsilon L} v$ with $\varepsilon > 0$ and $v \in \mathrm{dom}(L)$.
\end{enumerate}
Then we have
\begin{equation*}
  \lim_{|\tau|\rightarrow 0} P_{\Delta_1 \tau} \cdots P_{\Delta_N\tau} u = e^{-tL}u,
\end{equation*}
for any $u \in E$, where the limit goes over any sequence of partitions $\tau$ of the interval $[0, t]$ the mesh of which tends to zero.
\end{proposition}

We will subsequently prove the following result:

\begin{proposition} \label{PropProperFamily}
Set for $t>0$ and $u \in C^0(M, \V)$
\begin{equation*}
P_t u := P_{\{0<t\}} u,
\end{equation*}
where $\{0 < t\}$ is the trivial partition of the interval $[0, t]$ and the right hand side was defined in \eqref{DefPtau}. Furthermore, set $P_0u := u$. Then $P_t$ is a proper family on $C^0(M, \V)$ with the Laplace type operator $L$ as infinitesimal generator. Furthermore, $P_t$ extends uniquely to a proper family on $L^p(M, \V)$, for $1 \leq p < \infty$, with $L$ as infinitesimal generator. 
\end{proposition}

Using this proposition, we can prove the path integral formula above.

\begin{proof}[of Thm.~\ref{ThmPathIntegralNonSmooth}]

By Prop.~\ref{PropChernoff}, we have
\begin{equation*}
  \lim_{|\tau|\rightarrow 0} P_{\Delta_1 \tau} \cdots P_{\Delta_N\tau} u = e^{-tL} u,
\end{equation*}
where $P_t$ is the proper family from Prop.~\ref{PropProperFamily}. We now show by induction on the length $N$ of the partition that 
\begin{equation} \label{ClaimEquality}
P_{\Delta_1 \tau} \cdots P_{\Delta_N\tau} u = P_\tau u
\end{equation}
 for any partition $\tau = \{0 = \tau_0 < \tau_1 < \dots < \tau_N = t\}$ and any $u \in L^1(M, \V)$, where $P_\tau$ is the operator defined as in \eqref{DefPtau}. This is so for $N=1$ by definition. Suppose now that the result is true for some $N\geq 1$. For abbreviation, set
\begin{equation*}
  E(\gamma) := \frac{1}{4} \int_0^t \bigl|\dot{\gamma}(s)\bigr|^2 \dd s.
\end{equation*}
If then $\tau = \{0 = \tau_0 < \tau_1 < \dots < \tau_N = t\}$ is some partition of length $N$ and $\tau^\prime = \{ 0 = \tau^\prime_0 < \tau^\prime_1 < \dots < \tau_{N^\prime}^\prime\}$ is a partition of length $N^\prime \leq N$ (e.g.\ $N^\prime = 1$), then for $x \in M \setminus \partial M$, 
\begin{equation*}
\begin{aligned}
  P_\tau P_{\tau^\prime} u(x) &= (4\pi)^{-n(N+N^\prime)/2} \!\!\int_{H_{x;\tau}(M)} \int_{H_{\gamma(t);\tau^\prime}(M)} \!\!\!\!\!\!\!\!\!\!\!e^{-E(\gamma) - E(\gamma^\prime)} \mathcal{P}_B(\gamma)^{-1}\mathcal{P}_B(\gamma^\prime)^{-1}u\bigl(\gamma^\prime(t^\prime)\bigr) \dd \gamma^\prime \dd \gamma\\
 &= (4\pi)^{-n(N+N^\prime)/2}\!\! \int_{H_{x;\tau}(M)} \int_{H_{\gamma(t);\tau^\prime}(M)} \!\!\!\!\!\!\!\!\!\!\!e^{-E(\gamma*\gamma^\prime)} \mathcal{P}_B(\gamma*\gamma^\prime)^{-1}u\bigl((\gamma*\gamma^\prime)(t+t^\prime)\bigr) \dd \gamma^\prime \dd \gamma\\
 &= \fint_{H_{x;\tau*\tau^\prime}(M)}  \!\!\!\!\!\!\!\!\!e^{-E(\gamma)} \mathcal{P}_B(\gamma)^{-1}u\bigl(\gamma(t+t^\prime)\bigr)  \dd \gamma = P_{\tau*\tau^\prime}u(x),
\end{aligned}
\end{equation*}
where we always integrate with the respect to the discrete $H^1$ volume.
Here we used the multiplicativity of $\mathcal{P}_B(\gamma)$ (see \eqref{MultiplicativityPB}) and additivity $E(\gamma) + E(\gamma^\prime) = E(\gamma*\gamma^\prime)$ of the energy, as well as the co-area formula from Lemma~\ref{LemmaConcatenation}. A similar calculation can be made in the case $x \in \partial M$.

This shows that if \eqref{ClaimEquality} holds for partitions $\tau$ of length $N$, then it also holds for partitions $\tau$ of length less or equal to $2N$. In total, \eqref{ClaimEquality} holds for all partitions.
\end{proof}

The remainder of this section is dedicated to giving a proof of Prop.~\ref{PropProperFamily}. This is split up into several lemmas. We generally assume that we are in the setup of Thm.~\ref{ThmPathIntegralNonSmooth}, i.e.\ $L$ is a self-adjoint Laplace type operator with involutive boundary conditions $B$, acting on sections of a metric vector bundle $\V$ over a compact Riemannian manifold $M$ with boundary. By \eqref{LaplaceDecomposition}, we have $L = \nabla^*\nabla + V$ for a unique metric connection $\nabla$ on $\V$ and a symmetric endomorphism field $V \in C^\infty(M, \End(\V))$.

\begin{lemma}\label{LemmaGronwallForP}
  Let $\alpha$ be a bound on the pointwise operator norm of $V$. Then for the path-ordered integral $\mathcal{P}_B(\gamma)$ determined by $L$, we have
  \begin{equation*}
    |\mathcal{P}_B(\gamma)^{-1}| \leq e^{(b-a)\alpha}
  \end{equation*}
  where $\gamma: [a, b] \longrightarrow M$ is any absolutely continuous path and $|\,-\,|$ denotes the pointwise operator norm.
\end{lemma}

\begin{proof}
  Suppose first that $\gamma(s) \in M \setminus \partial M$ for $s \in (0, t)$, i.e.\ $\gamma$ does not reflect. Let $Q(s)$ be the solution to the ordinary differential equation \eqref{ODEforEinverse}. Then
  \begin{equation*}
  \begin{aligned}
    2 |Q(s)|\frac{\dd}{\dd s} |Q(s)| &= \frac{\dd}{\dd s} |Q(s)|^2 = 2\<Q(s), \nabla_s Q(s)\> = -2\<Q(s), V\bigl(\gamma(s)\bigr)Q(s)\> \\ 
    &\leq 2|Q(s)|\bigl|V\bigl(\gamma(s)\bigr)Q(s)\bigr|
     \leq 2\alpha |Q(s)|^2,
  \end{aligned}
  \end{equation*}
  hence 
  \begin{equation*}
    \frac{\dd}{\dd s} |Q(s)|\leq \alpha |Q(s)|.
  \end{equation*}
  From Gronwall's lemma \cite[10.5.1.3]{dieudonne}, we obtain therefore $|\mathcal{P}_B(\gamma)^{-1}| = |Q(t)| \leq e^{t\alpha}$. Now let $\gamma$ be arbitrary and let $\sigma_1 < \dots < \sigma_k$ be the times in $(a, b)$ that $\gamma$ hits the boundary. Then we have
  \begin{equation*}
  \begin{aligned}
    |\mathcal{P}_B(\gamma)^{-1}| &\leq \bigl|\mathcal{P}_B\bigl(\gamma|_{[a, \sigma_1]}\bigr)^{-1}\bigr|\bigl|\mathcal{P}_B\bigl(\gamma|_{[\sigma_1, \sigma_2]}\bigr)^{-1}\bigr| \cdots \bigl|\mathcal{P}_B\bigl(\gamma|_{[\sigma_{k-1}, \sigma_k]}\bigr)^{-1}\bigr|\bigl|\mathcal{P}_B\bigl(\gamma|_{[\sigma_k, b]}\bigr)^{-1}\bigr| \\
    &\leq e^{(\sigma_1-a) \alpha + (\sigma_2-\sigma_1) \alpha + \dots + (\sigma_k - \sigma_{k-1})\alpha + (b-\sigma_k) \alpha} = e^{(b-a)\alpha},
    \end{aligned}
  \end{equation*}
  where we used that $B$ is a self-adjoint involution, hence an isometry.
\end{proof}

Throughout the proof, we use the following notation.

\begin{notation}
Let $x \in M$, $t>0$ and $v \in \Omega_{t, x} \subseteq T_x M$.
\begin{enumerate}[(a)]
\item If $x \in M \setminus \partial M$, denote by $\gamma_v \in H_{x;\{0<t\}}^{\mathrm{refl}}(M)$ the unique reflected geodesic with $\dot{\gamma}_v(0) = v$ of length $t$. 
\item If $x \in \partial M$, and $v \in T_x^{>0}M = \{ v \mid \<v, \n\> > 0\}$ is inward directed, denote by $\gamma_v := (\gamma_v, +1) \in H_{x;\{0<t\}}^{\mathrm{refl}}(M)$ the unique reflected geodesic with $\dot{\gamma}_v(0) = v$ and positive sign.
\item If $x \in \partial M$, and $v \in T_x^{<0}M = \{ v \mid \<v, \n\> < 0\}$ is outward directed,  denote by $\gamma_v := (\gamma_v, -1) \in H_{x;\{0<t\}}^{\mathrm{refl}}(M)$ the unique reflected geodesic with $\dot{\gamma}_v(0) = Rv$ and negative sign.
\end{enumerate}
\end{notation}

Up to the sign in the case that $x \in \partial M$, $\gamma_v$ is the footpoint curve produced by the broken billiard flow. Now notice that we defined the smooth structure on $H_{x;\{0<t\}}^{\mathrm{refl}}(M)$ in such a way that the map
\begin{equation*}
  \Phi: T_x M \supseteq \Omega_{t, x} \longrightarrow H_{x;\{0<t\}}^{\mathrm{refl}}(M), ~~~~~~ v \longmapsto \gamma_v
\end{equation*}
is a diffeomorphism for any $x \in M$. The differential $d\Phi|_v$ assigns to a vector $w \in T_x M$ the Jacobi field $X_w$ along $\gamma_v$ with $X_w(0) = 0$ and $\nabla_s X_w(0+) = w$. Therefore
\begin{equation*}
  \bigl(d\Phi|_v w_1, d\Phi|_v w_2\bigr)_{\Sigma\text{-}H^1} = \<X_{w_1}(0+), X_{w_2}(0+)\> t = t\<w_1, w_2\>
\end{equation*}
so that $\Phi$ is a conformal mapping with
\begin{equation} \label{JacobianOfThisMapPhi}
  \bigl|\det\bigl(d \Phi|_v\bigr)\bigr| = t^{-n/2}.
\end{equation}
Because $|\dot{\gamma}_v(s)| \equiv |v|$ for all $s$ as ${\gamma}_v$ is a piecewise geodesic and $R$ is an isometry, we obtain
\begin{equation*}
  E(\gamma_v) = \frac{1}{2} \int_0^t |\dot{\gamma}_v(s)|^2 \dd s = \frac{t|v|^2}{2}.
\end{equation*}
Therefore, the transformation formula on the map $\Phi$ yields (using that $\Omega_{t, x}$ has full measure in $T_x M$  by Lemma~\ref{LemmaOpenAndDense}) that
\begin{equation} \label{ReformulationPt}
  P_t u(x) = \int_{T_x M} \varphi_t(v)\, \mathcal{P}(\gamma_v)^{-1} u\bigl(\gamma_v(t)\bigr)\dd v 
\end{equation}
where we set $\varphi_t(v) := t^{n/2}(4\pi)^{-n/2} e^{-t|v|^2/4}$. The function $\varphi_t$ is a simple Gaussian function, where the pre-factor just ensures that it integrates to one over $T_xM$.

\begin{lemma} \label{LemmaNormBound}
  Let $\alpha$ be a bound on the pointwise operator norm of $V$. Then for all $u \in C^0(M, \V)$ and for any $1 \leq p \leq \infty$, $t\geq 0$, we have
  \begin{equation*}
    \|P_t u\|_{L^p} \leq e^{\alpha t} \|u\|_{L^p}
  \end{equation*}
where $P_t$ is the family of Prop~\ref{PropProperFamily}.
\end{lemma}

Because $C^0(M, \V)$ is dense in $L^p(M, \V)$ if $p < \infty$, Lemma~\ref{LemmaNormBound} implies that $P_t$ extends uniquely to a family of bounded operators on $L^p(M, \V)$ satisfying the same norm bound for such $p$. In particular, $P_t$ satisfies property (i) of Prop.~\ref{PropChernoff} on each of the spaces $L^p(M, \V)$, $1 \leq p < \infty$.

In the proof and later, we denote by
\begin{equation*}
  \pi: TM \longrightarrow M
\end{equation*}
the canonical projection.

\begin{proof}
From \eqref{ReformulationPt} follows the estimate
\begin{equation*}
\begin{aligned}
  \|P_t u\|_\infty &\leq \sup_{x \in M} \int_{T_xM} \varphi_t(v) \,\bigl|\mathcal{P}_B(\gamma_v)^{-1}\bigr| \bigl|u\bigl(\gamma_v(t)\bigr)\bigr| \, \dd v \leq e^{t\alpha} \|u\|_\infty,
\end{aligned}
\end{equation*}
where we used that $|\mathcal{P}_B(\gamma_v)^{-1}| \leq e^{t\alpha}$ for all $v$ by Lemma~\ref{LemmaGronwallForP}, and the fact that the function $\varphi_t(v)$ integrates to one over $T_xM$. Hence the operator family $(P_t)_{t\geq 0}$ is uniformly bounded near zero on $C^0(M, \V)$.

For $1 \leq p < \infty$, we can pointwise use Jensen's inequality on the probability measure $\varphi_v(v)\, \dd v$ on $T_xM$ to obtain
\begin{equation*}
\begin{aligned}
  \|P_t u\|_{L^p}^p &\leq \int_{TM} \varphi_t(v) |\mathcal{P}_B(\gamma_v)^{-1}|^p |u\bigl(\gamma_v(t)\bigr)|^p \dd v \leq e^{tp\alpha}\int_{TM} \varphi_t(v) \bigl|\pi^*u\bigl(\Theta_t(v)\bigr)\bigr|^p \dd v
\end{aligned}
\end{equation*}
using the definition of the broken billiard flow. Now remember that the broken billiard flow preserves the measure on $TM$, as well as the norm of vectors, $|\Theta_{t}(v)| = |v|$, which implies $\varphi_t(v)= \varphi_t(\Theta_s(v))$ for all $s$. Hence transforming $v \mapsto \Theta_{-t}(v)$ gives
\begin{equation} \label{NormPreservingCalculation}
\begin{aligned}
  \int_{TM} \!\!\!\varphi_t(v) \bigl|\pi^*u\bigl(\Theta_t(v)\bigr)\bigr|^p \dd v
  &= \int_{TM}\!\!\! \varphi_t\bigl(\Theta_{-t}(v)\bigr) \bigl|\pi^*u(v)\bigr|^p \dd v = \int_{TM}\!\!\! \varphi_t(v) \bigl|\pi^*u(v)\bigr|^p \dd v \\
  &= \int_M \bigl|u(x)\bigr|^p \int_{T_xM} \!\!\!\varphi_t(v) \dd v \,\dd x = \|u\|_{L^p}^p.
\end{aligned}
\end{equation}
This shows the norm bound in the case $p< \infty$. 
\end{proof}

\begin{lemma} \label{LemmaQPreservesCZero}
 If $u \in C^0(M, \V)$, then also $P_tu \in C^0(M, \V)$, for all $t \geq 0$.
\end{lemma}

\begin{proof}
Choose a local trivialization $\psi: U \times \R^n \longrightarrow TM|_{U}$ over an open set $U\subseteq M$ that is an isometry in each fiber. Then since $\varphi(v) = \varphi(\psi_{x} v)$ for each $v \in \R^n$ and each $x \in U$,
\begin{equation*}
  P_tu(x) = \int_{\R^n} \varphi_t(v) \mathcal{P}_B(\gamma_{\psi_{x}v})^{-1}u\bigl(\gamma_{\psi_{x} v}(t)\bigr) \dd v.
\end{equation*}
If $x_j$ is a sequence in $U$ converging to $x \in U$ as $j \rightarrow \infty$, then $\psi_{x_j} v$ converges to $\psi_x v$ in the topology of $TM$. Therefore, by the Lebesgue's theorem of dominated convergence, it suffices to show that the function
\begin{equation*}
  f(t, v) := \mathcal{P}_B(\gamma_{v})^{-1}u\bigl(\gamma_{v}(t)\bigr)
\end{equation*}
is uniformly bounded and continuous in $v$ at almost all $v \in TM$. The function $u(\gamma_v(t))$ is continuous, since $u$ is continuous and $\gamma_v(t)$ depends continuously on $v \in \Omega_{t, x}$ (because the solutions of ordinary differential equations depend continuously on the initial data).

For the same reason, $\mathcal{P}_B(\gamma_{v})$ is continuous near all $v \in \Omega_{t, x}$ such that either $\pi(v) \notin \partial M$ or $v \in T^{>0}_xM = \{v \mid \< v, \n\> > 0\}$ for $x \in \partial M$.

It remains to check the case that $v \in T^{<0}_x M = \{v \mid \< v, \n\> < 0\}$ for $x \in \partial M$. To this end, let $v \in TM|_{\partial M}$ be outward directed and let $v_j \in TM$ be a sequence of vectors that converges to $v$. Let $0\leq\sigma_{1,j}< \dots < \sigma_{k, j}<t$ be the times when $\gamma_{v_j}$ hits the boundary (the number $k$ of hits stabilizes for $j$ large enough). Then
\begin{equation*}
\begin{aligned}
  \mathcal{P}_B(\gamma_{v_j})_B^{-1} = \mathcal{P}(\gamma_{v_j}|_{[0, \sigma_{1, j}]})^{-1} B &\mathcal{P}(\gamma_{v_j}|_{[\sigma_{1, j}, \sigma_{2, j}]})^{-1} B \cdots \\
  &~~~~~\cdots B \mathcal{P}(\gamma_{v_j}|_{[\sigma_{k-1, j}, \sigma_{k, j}]})^{-1} B \mathcal{P}(\gamma_{v_j}|_{[\sigma_{k, j}, t]})^{-1}
\end{aligned}
\end{equation*}
and if $0 = \sigma_1 < \dots < \sigma_k < t$ are the times when $\gamma_v$ hits the boundary, we have
\begin{equation*}
\begin{aligned}
  \mathcal{P}_B(\gamma_{v})_B^{-1} = B &\mathcal{P}(\gamma_{v}|_{[\sigma_{1}, \sigma_{2}]})^{-1} B \cdots B \mathcal{P}(\gamma_{v}|_{[\sigma_{k-1}, \sigma_{k}]})^{-1} B \mathcal{P}(\gamma_{v}|_{[\sigma_{k}, t]})^{-1}
\end{aligned}
\end{equation*}
Because $\sigma_{i,j}\rightarrow\sigma_i$ as $j \rightarrow \infty$ (in particular $\sigma_{1, j} \rightarrow 0$), $\mathcal{P}(\gamma_{v_j}|_{[0, \sigma_{1, j}]})^{-1}$ converges to the identity in this limit and hence $\mathcal{P}_B(\gamma_{v_j})^{-1}$ converges to $\mathcal{P}_B(\gamma_{v})^{-1}$.
\end{proof}

Lemma~\ref{LemmaQPreservesCZero} together with Lemma~\ref{LemmaNormBound} (for $p=\infty$) shows that $P_t$ preserves the space $C^0(M, \V)$ and that the family $(P_t)_{t\geq 0}$ satisfies property (i) of Prop.~\ref{PropChernoff} on this space. 

\medskip

From now on, we assume that we have  $V = 0$, that is $L = \nabla^*\nabla$ in the decomposition \eqref{LaplaceDecomposition}. Then $\mathcal{P}_B(\gamma) = [\gamma\|_0^t]_B$, the $B$-parallel transport, so that \eqref{ReformulationPt} reads
\begin{equation} \label{PreliminaryReformulation}
  P_t u(x) = \int_{T_xM} \varphi_t(v) [\gamma_v\|_0^{t}]_B^{-1} \pi^* u\bigl(\Theta_t(v)\bigr) \dd v,
\end{equation}
using the definition of the broken billiard flow.
Now substituting $v \mapsto v t^{-1/2}$, we obtain
\begin{equation*}
  P_t u(x) = \int_{T_xM} \varphi(v) [\gamma_v\|_0^{t^{1/2}}]_B^{-1} \pi^* u\bigl(\Theta_{t^{1/2}}(v)\bigr) \dd v,
\end{equation*}
where we set $\varphi(v) := \varphi_1(v)$ and used that $\Theta_t(sv) = s \Theta_{ts}(v)$ (which follows from \eqref{RescalingPropertyFlow}) and the fact that $\pi^*u(t^{-1/2}\Theta_{t^{1/2}}(v)) = \pi^*u(\Theta_{t^{1/2}}(v))$. This suggests defining
\begin{equation} \label{DefinitionQt}
  Q_t u(x) := P_{t^2}u(x) = \int_{T_xM} \varphi(v) [\gamma_v\|_0^{t}]_B^{-1} \pi^* u\bigl(\Theta_{t}(v)\bigr) \dd v
\end{equation}
for $u \in C^0(M, \V)$. Because $Q_t$ is just a rescaling of $P_t$, $Q_t$, $t\geq 0$ extends to a uniformly bounded family of operators just as $P_t$. Notice that $Q_t$ is actually well defined for all $t \in \R$, with $Q_0 = \id$.

\begin{lemma} \label{LemmaStrongContinuity}
In the case $V=0$, the operator family $(Q_t)_{t \in \R}$ (and hence also $(P_t)_{t \geq 0}$), is strongly continuous on $L^p(M, \V)$ for any $1 \leq p < \infty$.
\end{lemma}

\begin{proof}
We first show that for each $u \in C^0(M, \V)$ and any $x \in M$, the function $t \mapsto Q_tu(x)$ is continuous. 
For $u \in C^0(M, \V)$, consider the function
\begin{equation*}
f(t, v):= [\gamma_v\|_0^{t}]_B^{-1} \pi^* u\bigl(\Theta_{t}(v)\bigr)= [\gamma_v\|_0^{t}]_B^{-1} u\bigl(\gamma_v(t)\bigr).
\end{equation*}
for $t \in \R$, $v \in T_x M$. If for a given $t_0 \in \R$, we have $\gamma_v(t) \notin \partial M$ (which is the case for almost all $v$), then $f(t, v)$ is clearly continuous in $t$. Therefore $\varphi(v)f(t, v) \rightarrow \varphi(v) f(t_0, v)$ as $t \rightarrow t_0$ for almost all $v \in T_x M$. Since $\varphi(v) f(t, v) \leq \|u\|_\infty \varphi(v)$, we also found a dominating integrable function, hence
\begin{equation*}
  Q_t u(x) = \int_{T_xM} \varphi(v) f(t, v) \dd v \longrightarrow \int_{T_xM} \varphi(v) f(t_0, v) \dd v = Q_{t_0}u(x)
\end{equation*}
as $t \rightarrow t_0$, by the dominated convergence theorem. Furthermore, since $u \in C^0(M, \V)$, $Q_t u$ is uniformly bounded for $t$ in compact subsets of $\R$, by Lemma~\ref{LemmaNormBound}, and we have $Q_tu \rightarrow Q_{t_0}u$ pointwise almost everywhere as $t \rightarrow t_0$, hence also $Q_tu \rightarrow Q_{t_0}u$ in $L^p$, again by the dominated convergence theorem.

For a general $u \in L^p(M, \V)$, choose a family of continuous sections $u_k \in C^0(M, \V)$ such that $u_k \rightarrow u$ in $L^p$. Then
  \begin{equation*}
    \|Q_t u - Q_{t_0} u\|_{L^p} \leq \|Q_t (u - u_k)\|_{L^p} + \|Q_{t_0} (u_k - u)\|_{L^p} + \|Q_t u_k - Q_{t_0} u_k\|_{L^p}.
  \end{equation*}
  By the uniform boundedness of the family $(Q_t)_{t \in \R}$, one can now choose first $k$ large enough to make the first two terms as small as one likes and then $t$ close enough to $t_0$ to make the third term arbitrarily small. This shows that $Q_tu \rightarrow Q_{t_0}u$ as $t \rightarrow t_0$ in the general case, hence $(Q_t)_{r \in \R}$ is strongly continuous on $L^p(M, \V)$, for all $1 \leq p < \infty$.
\end{proof}

\begin{lemma} \label{LemmaStrongContinuity2}
 In the case $V=0$, the operator family $(Q_t)_{t \in \R}$ (and hence also $(P_t)_{t \geq 0}$), is strongly continuous on $C^0(M, \V)$. 
\end{lemma}

\begin{proof}
Fix $t \in \R$. For any $x \in M$, $s \in \R$ and  $u \in C^0(M, \V)$, we have
  \begin{equation*}
    \bigl|Q_t u(x) -Q_su(x)\bigr| \leq \int_{T_xM} \varphi(v) \bigl| [\gamma_v\|_0^{t}]_B^{-1} u\bigl(\gamma_v(t)\bigr) - [\gamma_v\|_0^{s}]_B^{-1} u\bigl(\gamma_v(s)\bigr)\bigr| \dd v.
  \end{equation*}
  Let $\varepsilon >0$ and choose $R>0$ so large that
  \begin{equation*}
   2 \|u\|_\infty \int_{B_R(0)^c} \varphi(v)  \dd v \leq \frac{\varepsilon}{2},
  \end{equation*}
  where $B_R(0)$ denotes the $R$-ball around zero in $T_xM$.
  Now because $u \in C^0(M, \V)$ and $M$ is compact, $u$ is uniformly continuous (meaning that in local trivialization of $\V$, each component is uniformly continuous). Therefore, there exists $\delta>0$ such that
  \begin{equation*}
    \bigl| [\gamma_v\|_0^{t}]_B^{-1} u\bigl(\gamma_v(t)\bigr) - [\gamma_v\|_0^{s}]_B^{-1} u\bigl(\gamma_v(s)\bigr)\bigr| \leq \frac{\varepsilon}{2}\left( \int_{B_R(0)} \varphi(v)  \dd v\right)^{-1}
  \end{equation*}
  for all $v \in TM$ with $|v| \leq R$ and all $s$ with $|t-s| \leq \delta$. Because
\begin{equation*}
  \bigl| [\gamma_v\|_0^{t}]_B^{-1} u\bigl(\gamma_v(t)\bigr) - [\gamma_v\|_0^{s}]_B^{-1} u\bigl(\gamma_v(s)\bigr)\bigr| \leq 2 \|u\|_{\infty}
\end{equation*}  
  as $[\gamma_v\|_0^{t}]_B^{-1}$ is a fiberwise isometry,  we obtain in total that
  \begin{equation*}
  \begin{aligned}
    &\bigl|Q_t u(x) -Q_su(x)\bigr|\\
    &~~~~~~\leq \int_{B_R(0)} \varphi(v) \bigl| [\gamma_v\|_0^{t}]_B^{-1} u\bigl(\gamma_v(t)\bigr) - [\gamma_v\|_0^{s}]_B^{-1} u\bigl(\gamma_v(s)\bigr)\bigr| \dd v + 2 \|u\|_\infty \int_{B_R(0)^c} \varphi(v)\\
    &~~~~~~\leq \frac{\varepsilon}{2}+ \frac{\varepsilon}{2} \leq \varepsilon.
    \end{aligned}
  \end{equation*}
   for all $x \in M$, whenever $|t-s| \leq \delta$. The lemma follows.
\end{proof}

\begin{lemma} \label{LemmaPointwiseC2}
  Let $u \in C^2_B(M, \V)$, meaning that $u$ is a $C^2$ section of $\V$ satisfying the involutive boundary condition given by $B$. Then for each $x \in M$, the function $t \mapsto Q_tu(x)$ is $C^2$ and we have
  \begin{equation*}
  \begin{aligned}
    Q_t^\prime u(x) &= \int_{T_x M} \varphi(v) [\gamma_v\|_0^t]_B^{-1} \nabla_{\Theta_t(v)} u\bigl(\gamma_v(t)\bigr)\, \dd v,\\
    Q_t^{\prime\prime} u(x) &= \int_{T_x M} \varphi(v) [\gamma_v\|_0^t]_B^{-1}  \nabla^2 u|_{\gamma_v(t)}\bigl[\Theta_t(v), \Theta_t(v)\bigr]\, \dd v 
  \end{aligned}
  \end{equation*}
\end{lemma}

\begin{proof}
Set as before
\begin{equation*}
    f(t, v) := [\gamma_v\|_0^{t}]^{-1}_B \pi^*u \bigl(\Theta_t(v)\bigr).
\end{equation*}
We first show that given $v \in T_xM$, $f(t, v)$ is $C^{1, 1}$ on the interval $[0, T(v))$, where $T(v)$ is the maximal life-time of the reflected geodesic $\gamma_v$ with $\dot{\gamma}_v(0+) = v$ (respectively $\dot{\gamma}_v(0+) = Rv$ if $v \in TM|_{\partial M}$ and $v$ is outward directed). Let $0 \leq \sigma_1 < \sigma_2< \dots < T(v)$ be the times in this interval where $\gamma_v$ hits the boundary (these are finitely many if $T(v)$ is finite, but may be infinitely many otherwise). Then clearly, $f(t, v)$ is $C^2$ on $[0, T(v)) \setminus \{\sigma_1, \sigma_2, \dots \}$ with 
\begin{equation*}
\begin{aligned}
  \frac{\partial f}{\partial t}(t, v) &= [\gamma_v\|_0^{t}]^{-1}_B \nabla_{\Theta_t(v)} u\bigl( \gamma_v(t)\bigr),\\
  \frac{\partial^2 f}{\partial t^2}(t, v) &= [\gamma_v\|_0^{t}]^{-1}_B\Bigl(\nabla^2 u|_{\gamma_v(t)} \bigl[\Theta_t(v), \Theta_t(v)\bigr] + \nabla_{\nabla_t \Theta_t(v)}u\bigl( \gamma_v(t)\bigr)\Bigr)
\end{aligned}
\end{equation*}
However, $\nabla_t \Theta_t(v) = 0$, which follows from the fact that $\Theta_t(v)$ is the velocity vector field of a geodesic. We need to check continuity of the derivatives at the times $\sigma_j$. Decompose $\dot{\gamma}_v(\sigma_j+) = w^\prime + w_0 \n$ with $w^\prime \in T_{\gamma(\sigma_j)}\partial M$ and $w_0 \in \R$, so that $\dot{\gamma}_v(\sigma_j-) = w^\prime - w_0 \n$. Then because $u$ satisfies the boundary condition, we have $u|_{\partial M} \in C^\infty(\partial M, \mathcal{W}^+)$ and $\nabla_\n u|_{\partial M} \in C^\infty(\partial M, \mathcal{W}^-)$, hence
\begin{equation*}
  Bu\bigl(\gamma_v(\sigma_j)\bigr) = u\bigl(\gamma_v(\sigma_j)\bigr), ~~~~~~~ B\nabla_\n u\bigl(\gamma_v(\sigma_j)\bigr) = - \nabla_\n u\bigl(\gamma_v(\sigma_j)\bigr),
\end{equation*}
and
\begin{equation} \label{BCalculation}
\begin{aligned}
  B\nabla_{\dot{\gamma}_v(\sigma_j+)} u\bigl( \gamma_v(\sigma_j)\bigr) &= 
  B\nabla_{w^\prime} u\bigl( \gamma_v(\sigma_j)\bigr) + w_0 B\nabla_{\n} u\bigl( \gamma_v(\sigma_j)\bigr)\\
  &=\nabla_{w^\prime} u\bigl( \gamma_v(\sigma_j)\bigr) - w_0 \nabla_{\n} u\bigl( \gamma_v(\sigma_j)\bigr)\\
  &=\nabla_{\dot{\gamma}_v(\sigma_j-)} u\bigl( \gamma_v(\sigma_j)\bigr).
\end{aligned}
\end{equation}
For the second equality, notice that if $\eta: (-\varepsilon, \varepsilon) \longrightarrow \partial M$ with $\dot{\eta}(0) = w^\prime$, then
\begin{equation*}
  \nabla_{w^\prime} u\bigl( \gamma_v(\sigma_j)\bigr) = \nabla_s\bigr|_{s=0}\bigl\{ u\bigl(\eta(s)\bigr) \bigr\}\in \mathcal{W}^+,
\end{equation*}
since $u(\eta(s)) \in \mathcal{W}^+$ for each $s$ and the splitting is parallel by assumption. Hence indeed $B\nabla_{w^\prime} u\bigl( \gamma_v(\sigma_j)\bigr) = \nabla_{w^\prime} u\bigl( \gamma_v(\sigma_j)\bigr)$. Now by \eqref{BCalculation} and the definition of $[\gamma_v\|_0^{\sigma_j+}]^{-1}_B$, we have
\begin{equation*}
\begin{aligned}
  \frac{\partial f}{\partial t}(\sigma_j+, v) &= [\gamma_v\|_0^{\sigma_j+}]^{-1}_B \nabla_{\dot{\gamma}_v(\sigma_j+)} u\bigl( \gamma_v(\sigma_j)\bigr)\\
  &= [\gamma_v\|_0^{\sigma_j+}]^{-1}_B B\nabla_{\dot{\gamma}_v(\sigma_j-)} u\bigl( \gamma_v(\sigma_j)\bigr)\\
  &=  [\gamma_v\|_0^{\sigma_j-}]^{-1}_B \nabla_{\dot{\gamma}_v(\sigma_j+)} u\bigl( \gamma_v(\sigma_j)\bigr)
  =\frac{\partial f}{\partial t}(\sigma_j-, v)
\end{aligned}
\end{equation*}
so that the derivative is indeed continuous. 

To check that the derivative of $f$ is Lipschitz, notice that
\begin{equation*}
  \left|\frac{\partial^2 f}{\partial t^2}(t, v)\right| \leq \|\nabla^2 u\|_\infty |v|^2 =:\ell(v)
\end{equation*}
so that $\frac{\partial f}{\partial t}(t, v)$ is uniformly Lipschitz with Lipschitz constant $\ell(v)$.
Now because the function $\varphi(v)f(t, v)$ is $C^1$ in $t$ for almost all $v$, with integrable derivative, we may differentiate under the integral sign to obtain
\begin{equation*}
  Q^\prime u(x) = \int_{T_xM} \varphi(v) \,\frac{\partial f}{\partial t}(t, v)\, \dd v = \int_{T_x M} \varphi(v) [\gamma_v\|_0^t]_B^{-1} \nabla_{\Theta_t(v)} u\bigl( \gamma_v(t)\bigr)\, \dd v.
\end{equation*}
For the second derivative, note that we have
\begin{equation*}
  \lim_{\varepsilon\rightarrow 0} \frac{1}{\varepsilon}\left(\frac{\partial f}{\partial t}(t+\varepsilon, v) - \frac{\partial f}{\partial t}(t, v)\right) = \frac{\partial^2 f}{\partial t^2}(t, v)
\end{equation*}
for almost all $v$ (since for fixed $t$, $\gamma_v(t) \notin \partial M$ for almost all $v$) and
\begin{equation*}
  \left|\frac{1}{\varepsilon}\left(\frac{\partial f}{\partial t}(t+\varepsilon, v) - \frac{\partial f}{\partial t}(t, v)\right)\right| \leq \ell(v)
\end{equation*}
by the considerations before.
Hence $\varphi(v) \ell(v)$ is an integrable dominating function for the difference quotient, and
\begin{equation*}
\begin{aligned}
  Q^{\prime\prime}_t u(x) &= \lim_{\varepsilon \rightarrow 0} \int_{T_xM} \varphi(v)\frac{1}{\varepsilon}\left(\frac{\partial f}{\partial t}(t+\varepsilon, v) - \frac{\partial f}{\partial t}(t, v)\right) \dd v \\
  &=   \int_{T_xM} \varphi(v) \lim_{\varepsilon \rightarrow 0} \frac{1}{\varepsilon}\left(\frac{\partial f}{\partial t}(t+\varepsilon, v) - \frac{\partial f}{\partial t}(t, v)\right)\dd v\\
   &=   \int_{T_xM}  \varphi(v) \frac{\partial^2 f}{\partial t^2}(t, v) \dd v,
\end{aligned}
\end{equation*}
where the exchange of integration and taking the limit is justified by the dominated convergence theorem. Continuity of $Q^{\prime\prime}_tu(x)$ in $t$ can be shown just as in the proof of Lemma~\ref{LemmaStrongContinuity}.
\end{proof}

\begin{proof}[of Prop.~\ref{PropProperFamily}]
The proof consists of two steps.

{\em Step 1.} Assume that $V = 0$ so that $P_t = Q_{t^{1/2}}$, with $Q_t$ given by \eqref{DefinitionQt}. In this case we already know from the Lemmas \ref{LemmaNormBound} and \ref{LemmaStrongContinuity}, respectively Lemma~\ref{LemmaStrongContinuity2} that $P_t$ satisfies properties (i)-(ii) of Prop.~\ref{PropChernoff} on each of the spaces $L^p(M, \V)$ with $1 \leq p < \infty$ and $C^0(M, \V)$, so it remains to verify property (iii). To this end, for $u \in C^2_B(M, \V)$, notice that
\begin{equation*}
  Q_0^\prime u(x) = \int_{T_xM} \varphi(v) \nabla_v u(x) \dd v = 0,
\end{equation*}
since the integrand is an odd function. Therefore, pointwise Taylor expansion yields
\begin{equation} \label{TaylorExpansionQt}
  Q_tu(x) = u(x) + \int_0^t (t-s)Q^{\prime\prime}_su(x)\dd s = u(x) + t^2 \int_0^1 (1-s) Q^{\prime\prime}_{ts} u(x) \dd s.
\end{equation}
The Taylor expansion is justified since $t \mapsto Q_t u(x)$ is $C^2$ for all $x \in M$, by Lemma \ref{LemmaPointwiseC2}. Formula \eqref{TaylorExpansionQt} implies that for each $x \in M$, we have
\begin{equation} \label{DifferenceQuotient}
  \frac{1}{t} \bigl(P_t u(x) - u(x)\bigr) = \int_0^1 (1-s) Q^{\prime\prime}_{t^{1/2} s}u(x) \dd s,
\end{equation}
so that limit evaluates to
\begin{equation*}
  \lim_{t \rightarrow 0}\frac{1}{t} \bigl(P_t u(x) - u(x)\bigr) = \int_0^1 (1-s) Q^{\prime\prime}_{0}u(x) \dd s
  = \frac{1}{2} Q^{\prime\prime}_0u(x),
\end{equation*}
as $t \mapsto Q^{\prime\prime}_tu(x)$ is continuous. Here we have
\begin{equation*}
  Q^{\prime\prime}_0u(x) = \int_{T_xM} \varphi(v) \nabla^2u|_x[v, v] \dd v = 2 \,\tr \,\nabla^2u|_x = -2L u(x),
\end{equation*}
where the second equality is an elementary result for Gaussian integrals.
Hence for any $x \in M$.
\begin{equation*}
  \lim_{t \rightarrow 0}\frac{1}{t} \bigl(P_t u(x) - u(x)\bigr) = - L u(x).
\end{equation*}
Furthermore, one shows similarly to the proof of Lemma~\ref{LemmaStrongContinuity2} that the convergence here is even uniform in $x$, so that this convergence is true in the spaces $C^0(M, \V)$ and $L^p(M, \V)$. By parabolic regularity up to the boundary, we have $e^{-tL} v \in C^\infty_B(M, \V)$ for each $v \in L^p(M, \V)$, so it indeed suffices to check this limit for $u \in C^2_B(M, \V)$. This proves property (iii) in the case that $V = 0$.

{\em Step 2.} For the case that $V \neq 0$, we use the Taylor expansion
\begin{equation*}
  \mathcal{P}_B(\gamma)^{-1} = [\gamma\|_0^t]^{-1}_B - \int_0^t \mathcal{P}_B(\gamma|_{[0, s]})^{-1}V\bigl(\gamma(s)\bigr) [\gamma\|_s^t]_B^{-1} \dd s.
\end{equation*}
Let $P_t$ be defined as in the proposition for the operator $L=\nabla^*\nabla + V$ and write $\tilde{P}_t$ for the operator family corresponding to the operator $\tilde{L} := \nabla^*\nabla$. Then by \eqref{ReformulationPt}, we have for $u \in C^0(M, \V)$
\begin{equation*}
  P_tu(x) = \tilde{P}_tu(x) - \int_0^t \int_{T_x M} \varphi_t(v) \mathcal{P}_B\bigl(\gamma_v|_{[0, s]}\bigr)^{-1}V\bigl(\gamma_v(s)\bigr) [\gamma_v\|_s^t]^{-1}_B u\bigl(\gamma_v(t)\bigr)\,\dd v \,\dd s.
\end{equation*}
Setting $\alpha := \|V\|_\infty$, Jensen's inequality and Lemma~\ref{LemmaGronwallForP} imply
\begin{equation*}
\begin{aligned}
  \|P_t u- \tilde{P}_t u\|_{L^p}^p &= t^p\int_M \left|\frac{1}{t}\int_0^t \int_{T_xM} \varphi_t(v) \mathcal{P}_B\bigl(\gamma_v|_{[0, s]}\bigr)^{-1}V\bigl(\gamma_v(s)\bigr) [\gamma_v\|_s^t]^{-1}_B u\bigl(\gamma_v(t)\bigr) \dd v\,\dd s\right|^p \dd x\\
  &\leq  t^{p-1}\int_0^t  \alpha e^{\alpha s} \int_{TM} \varphi_t(v) \bigl|u\bigl(\gamma_v(t)\bigr)\bigr|^p\dd v\,\dd s = \|u\|_{L^p}^p t^{p-1} (e^{\alpha t} - 1)
\end{aligned}
\end{equation*}
where in the last step, we used the calculation \eqref{NormPreservingCalculation}. This shows that $P_t - \tilde{P}_t$ converges to zero in norm as $t \rightarrow 0$, hence $P_t$ is strongly continuous at zero on $L^p$ (since $\tilde{P}_t$ is, by Lemma~\ref{LemmaStrongContinuity}). For the $C^0$ case, we similarly find $\|P_t u- \tilde{P}_t u\|_{C^0} \leq \alpha t e^{\alpha t}\|u\|_{C^0}$, so $P_t$ is also strongly continuous at zero on $C^0$ (by virtue of Lemma~\ref{LemmaStrongContinuity2}). Strong continuity near $t_0 \neq 0$ can be shown similar as before, by using the fact that the integrand
\begin{equation*}
  \int_0^t \varphi_t(v) \mathcal{P}_B\bigl(\gamma_v|_{[0, s]}\bigr)^{-1}V\bigl(\gamma_v(s)\bigr) [\gamma_v\|_s^t]^{-1}_B u\bigl(\gamma_v(t)\bigr)\,\dd v
\end{equation*}
depends continuously on $t$ near $t_0 \neq 0$ for almost all $v \in T_x M$ and has a dominating integrable function.

It remains to check that $P_t$ has the correct infinitesimal generator. From the Taylor expansion above follows that
\begin{equation*}
\begin{aligned}
\frac{1}{t} \bigl(P_t u(x) - u(x)\bigr) = \frac{1}{t} \bigl(&\tilde{P}_t u(x) - u(x)\bigr)\\
&- \frac{1}{t}\int_0^t \int_{T_xM} \varphi_t(v) \mathcal{P}_B\bigl(\gamma_v|_{[0, s]}\bigr)^{-1}V\bigl(\gamma_v(s)\bigr) [\gamma_v\|_s^t]^{-1}_B u\bigl(\gamma_v(t)\bigr)\,\dd v\, \dd s
\end{aligned}
\end{equation*}
The first term converges to $-\tilde{L}u(x) = -\nabla^*\nabla u(x)$ uniformly by Step 1, while the second term converges uniformly to $-Vu(x)$, which can be shown similar to the proof of Lemma~\ref{LemmaStrongContinuity2}. This finishes the proof in the general case.
\end{proof}

    \bibliography{Literatur} 
 
\end{document}